\documentclass[runningheads]{llncs}
\usepackage[T1]{fontenc}
\usepackage{graphicx}
\usepackage[usenames,dvipsnames]{color}
\usepackage{amsmath}
\usepackage{amssymb}
\usepackage{cleveref}
\usepackage{tikz}
\usepackage{tikz-3dplot}
\usepackage{tikz-cd}
\usepackage{pgfplots}
\usepackage{mathtools}
\usepackage{commath} 
\usepackage{stmaryrd} 

\def\S{\mathbb{S}}

\def\NN{\mathbb{N}}

\def\RR{\mathbb{R}}
\def\CC{\mathbb{C}}
\def\Slr{\mathbb{S}\left(\ell_{\RR}^2 \right)}

\def\U{\mathcal{U}}
\def\Glt{\mathcal{G}^{\ell^2}}
\def\Gfr{\mathcal{G}^{\mathrm{FR}}}
\def\i{\mathrm{i}}

\def\d{\mathrm{d}}

\newcommand{\R}{\mathbb{R}}

\begin{document}
\title{Information Geometry on the $\ell^2$-Simplex via the $q$-Root Transform}
%
%
\author{
	Levin Maier\inst{1}$^{\ast}$
}
\renewcommand{\thefootnote}{\fnsymbol{footnote}}
\footnotetext[1]{Corresponding author}
\renewcommand{\thefootnote}{\arabic{footnote}}
\authorrunning{L. Maier}
%
\institute{%
	Heidelberg University, Im Neuenheimer Feld 205, 69120 Heidelberg, Germany\\
	\email{lmaier@mathi.uni-heidelberg.de}
}
\maketitle              
\begin{abstract}
In this paper, we introduce \emph{$\ell^p$-information geometry}, an infinite dimensional framework that shares key features with the geometry of the space of probability densities \( \mathrm{Dens}(M) \) on a closed manifold, while also incorporating aspects of measure-valued information geometry. We define the \emph{$\ell^2$-probability simplex} with a noncanonical differentiable structure induced via the \emph{$q$-root transform} from an open subset of the $\ell^q$-sphere. This structure renders the $q$-root map an \emph{isometry}, enabling the definition of \emph{Amari--Čencov $\alpha$-connections} in this setting.

We further construct \emph{gradient flows} with respect to the $\ell^2$--Fisher--Rao metric, which solve an infinite-dimensional linear optimization problem. These flows are intimately linked to an \emph{integrable Hamiltonian system} via a \emph{momentum map} arising from a Hamiltonian group action on the infinite-dimensional complex projective space.

\end{abstract}


\section{Introduction}\label{sec:introduction}
Infinite-dimensional information geometry—particularly the geometry of the space of smooth probability densities $\mathrm{Dens}(M)$ on a closed Riemannian manifold equipped with the so-called Fisher--Rao metric—has received considerable attention; see \cite{FisherRaoisunique,bauerLpFisherRaoMetricAmari2024,khesin_GAFA,khesin_geometry_2019} and the references therein. Another possible generalization of finite-dimensional information geometry, as studied in \cite{MR1800071,InfGeo_book} and the references therein, to the infinite-dimensional setting would be the Fisher--Rao geometry of infinite-dimensional probability simplices, where each point has infinitely many coordinate entries. A canonical candidate for this is the probability simplex in $\ell^2_{\mathbb{R}}$, called $\ell^2$-\emph{probability simplex} introduced in \Cref{s: 2}. Surprisingly, when equipped with its canonical differentiable structure inherited from the ambient space, the Fisher--Rao metric  fails to be well-defined. We overcome this issue by pulling back, via the so-called $q$-root transform \eqref{e: def q-root transform}, the differentiable structure of an open subset of the $\ell^q$ sphere to the probability simplex in $\ell^2$, which turns out to be stronger than rapid decay conditions on the sequences. This allows us to define the $\ell^2$-Fisher--Rao metric \eqref{e: defi l2 Fisher-Rao}—the $\ell^2$ analogue of the Fisher--Rao metric—as well as the $\ell^q$-Fisher--Rao metric \eqref{e: def q-Fisher-Rao metric} on the $\ell^2$-probability simplex. The latter serves as the $\ell^q$ analogue of the $L^q$-Fisher--Rao metric introduced in \cite{bauerLpFisherRaoMetricAmari2024}. 
Moreover, even more is true: when equipped with one of these metrics, the $\ell^2$-probability simplex is isometric to an open subset of the $\ell^q$ sphere equipped with the $\ell^q$ metric; see \Cref{t: square root map is isometry} for the $\ell^2$ case and \Cref{T: q-root transform as isom} for the $\ell^q$ case. These correspond to the $\ell^2$ and $\ell^q$ versions, respectively, of \cite[Thm.~3.1]{khesin_GAFA} and \cite[Thm.~4.10]{bauerLpFisherRaoMetricAmari2024}. Before moving on, the author hopes that the $\ell^2$-probability simplex, equipped with this pullback differentiable structure, can serve as a toy model for the Fisher--Rao geometry of the space of probability densities on non-compact, unbounded Riemannian manifolds.
As a first illustration of the previously developed $\ell^2$-information geometry, we extend in \Cref{s: 3} and \Cref{S: 4} the connection between the information-geometric perspective on linear programming problems and the totally integrable Hamiltonian system introduced in \cite{faybusovich_hamiltonian_1991} to the infinite-dimensional setting of $\ell^2$ geometry. Specifically, in \Cref{s: 3}, we solve an infinite-dimensional linear programming problem in the $\ell^2$ setting via a gradient flow with respect to the $\ell^2$-Fisher–Rao metric. In \Cref{t: integrability of Hamiltonian system}, we show that this gradient flow originates from a totally integrable, infinite-dimensional Hamiltonian system defined on an infinite-dimensional Kähler manifold. 
\section{Introducing $\ell^2$ and $\ell^q$-Information Geometry}\label{s: 2}
We begin by introducing the setting in detail. We denote by $\left(\ell^2_\mathbb{R}, \langle\cdot, \cdot\rangle_{\ell^2}\right)$ the space of real-valued sequences $(x_n)_{n\in \mathbb{N}} \subset \mathbb{R}$ satisfying $\sum_{n=0}^{\infty} x_n^2 < \infty.$
For convenience, we will later denote such sequences simply by $(x_n)$. This space is equipped with the inner product $\langle (x_n), (y_n)\rangle_{\ell^2} = \sum_{n=0}^{\infty} x_n y_n.$
So the unit sphere in $\ell^2_{\mathbb{R}}$ is 
\[
\mathbb{S}\left(\ell^2_{\mathbb{R}} \right)=\left\{(x_n)\in \ell_{\mathbb{R}}^2 : \sum_{n=0}^{\infty} x_n^2 = 1 \right\},
\]
with the tangent space at the point $(x_n) \in \mathbb{S}(\ell^2_{\mathbb{R}})$ given by
\[
T_{(x_n)} \mathbb{S}(\ell^2_{\mathbb{R}}) = \left\{(v_n) \in \ell_{\mathbb{R}}^2 : \langle (x_n), (v_n) \rangle_{\ell^2_{\mathbb{R}}} = 0 \right\}.
\]
In this setting, the round metric $\Glt$, respectively the $\ell^2$-metric, is precisely the restriction of $\langle\cdot, \cdot \rangle_{\ell^2}$ onto $\mathbb{S}(\ell^2_{\mathbb{R}})$. It is well known that \( (\mathbb{S}(\ell^2_{\mathbb{R}}), \Glt) \) is a strong Riemannian Hilbert manifold in the sense of \cite{LA99}; that is, \( \Glt \) induces a bundle isomorphism 
\(T\mathbb{S}(\ell^2_{\mathbb{R}}) \rightarrow T^*\mathbb{S}(\ell^2_{\mathbb{R}})\).
The open subset of strictly positive sequences is denoted by 
\begin{equation}\label{e: definition Ul2 }
	\U:=\left\{(x_n)\in \Slr:\  x_n>0 \quad \forall n\in \NN\right\},
\end{equation}
is as an open subset of $(\mathbb{S}(\ell^2_{\mathbb{R}}), g)$ equipped with the round metric $\Glt$ also an strong Riemannian Hilbert manifold. 
We move on by introducing the $\ell^2_{\mathbb{R}}$-analogue of the space of probability densities, the $\ell^2_{\mathbb{R}}$-\emph{probability simplex}, defined by 
\begin{equation}\label{e: definition l2probability simplex}
	\Delta := \left\{(p_n) \in \ell^1_{\mathbb{R}} : \sum_{n=0}^{\infty} p_n = 1 \quad \text{and} \quad p_n > 0 \quad \forall n \in \mathbb{N} \right\}.
\end{equation}
Note, that at this point, we have several choices for how to equip $\Delta$ with a differentiable structure. We choose the differentiable structure on $\Delta$ so that the following homeomorphism, called the \emph{square-root map}, becomes a diffeomorphism:
\[
	\varPhi: \Delta \longrightarrow \mathcal{U}, \quad (p_n) \mapsto (\sqrt{p_n}).
\]
This structure is completely different from the differentiable structure induced by the ambient space $\ell^2_{\mathbb{R}}$. The tangent space to $\Delta$ at a point $(p_n) \in \Delta$, with respect to this differentiable structure, is given by
\begin{equation}\label{d: differentiable structure}
    T_{(p_n)}\Delta = \left\{ (v_n) \in \ell^1_{\mathbb{R}} : \sum_{n=0}^{\infty} v_n = 0 \quad \text{and} \quad \left(\frac{v_n}{\sqrt{p_n}}\right) \in \ell^2_{\mathbb{R}} \right\}.
\end{equation}
By introducing the $\ell^2$-\emph{Fisher--Rao} metric as:
\begin{equation}\label{e: defi l2 Fisher-Rao}
	\mathcal{G}^{\mathrm{FR}}_{(p_n)}\left((v_n), (w_n) \right) := \frac{1}{4}\sum_{n=0}^{\infty} \frac{v_n \cdot w_n}{p_n}, \quad \forall (v_n), (w_n) \in T_{(p_n)}\Delta,
\end{equation}
the space $(\Delta, \mathcal{G}^{\mathrm{FR}})$ becomes a strong Riemannian Hilbert manifold. Before we move on, we note the following:

\begin{remark}\label{r: the right diff strcuture}
A condition similar to the one on \( \left(\frac{v_n}{\sqrt{p_n}}\right) \) in \eqref{d: differentiable structure} also appears in \cite[§2]{Friedrich1991}. Without this condition, the \( \ell^2 \)-Fisher–Rao metric \eqref{e: defi l2 Fisher-Rao} is not well-defined. Even under decay assumptions analogous to those in \cite[§3.5]{MichorMumfordZoo}, one can construct rapidly decaying sequences in \( \Delta \) such that the \( \ell^2 \)-Fisher–Rao metric fails to be finite on all tangent vectors. 
\end{remark}
Subject to our chosen differentiable structure in \eqref{d: differentiable structure}, and by a computation similiar to \cite[Thm. 3.1]{khesin_GAFA}, it follows directly that:

\begin{theorem}\label{t: square root map is isometry}
The square-root map \( \varPhi \) defined by
\[
\varPhi: (\Delta, \mathcal{G}^{\mathrm{FR}}) \longrightarrow (\mathcal{U}, \mathcal{G}^{\mathrm{L}^2}), \quad (p_n) \mapsto (\sqrt{p_n}),
\]
is an isometry.
\end{theorem}

\begin{remark}
This result can be seen as a blend of \cite[Thm.~3.1]{khesin_GAFA} and the methods in \cite{Friedrich1991}, but does not follow immediately from them. Accordingly, \Cref{t: square root map is isometry} provides yet another infinite-dimensional analogue of \cite[Proposition~2.1]{InfGeo_book}, and it relies crucially on the differentiable structure chosen in \eqref{d: differentiable structure}, as emphasized in \Cref{r: the right diff strcuture}.
Interestingly, finite-dimensional probability simplices $\Delta^N $ cannot be embedded as totally geodesic submanifolds into the infinite-dimensional simplex \( \Delta \). Indeed, \( \Delta^N \subseteq \partial \Delta \), which illustrates how \( \ell^2_{\mathbb{R}} \)-information geometry differs fundamentally from its finite-dimensional counterpart.
Moreover, an element of \( \Delta \) can be interpreted as a discretization of a probability measure subject to infinitely many measurements.
Note that \Cref{t: square root map is isometry} is effectively illustrated by the \Cref{fig:sqrt-map}.
\end{remark}
\begin{figure}[ht]
	\centering
\includegraphics{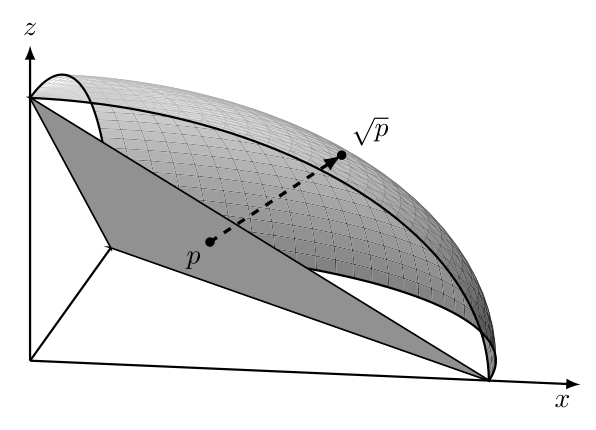}
	\caption{Illustration of the square-root map as an isometry.}
	\label{fig:sqrt-map}
\end{figure}
A natural question that arises for readers familiar with infinite-dimensional information geometry is the following:  
What is the analogue of \cite{FisherRaoisunique} in the setting of $\ell^2$ information geometry?

As a first illustration of \Cref{t: square root map is isometry}, we observe that $(\Delta, \Gfr)$ is geodesically convex; that is, for every pair of points, there exists a length-minimizing geodesic connecting them. This follows directly from the well-known fact that $(\Slr, \Glt)$ is geodesically convex, and consequently, $(\U, \Glt)$ is also geodesically convex.  
Therefore, by \Cref{t: square root map is isometry}, we obtain:
\begin{corollary}
	The $\ell^2$-probability simplex $\Delta$, equipped with the $\ell^2$-Fisher–Rao metric $\Gfr$, is geodesically convex. 
\end{corollary}
\begin{remark}
	At this point, we note that the Hopf–Rinow theorem is quite delicate in the context of infinite-dimensional Riemannian Hilbert manifolds, as it does not hold in full generality \cite{HopfrinowfalseAktkin},\cite{HopfrinowfalseEkeland}.  
	Moreover, geodesic completeness does not necessarily imply geodesic convexity. Interestingly, the latter does hold for half-Lie groups equipped with a strong right-invariant Riemannian metric \cite{Bauer_2025}.
\end{remark}
We conclude this section by presenting the $\ell^q_{\RR}$-analogue of the recent and intriguing work of Bauer–Le Brigant–Lu–Maor~\cite{bauerLpFisherRaoMetricAmari2024}. Here, $\ell^q_{\RR}$ denotes the space of all real sequences $(x_n)$ such that the $\ell^q_{\RR}$-norm $\Vert \cdot \Vert_{\ell^q}$ is finite, i.e.,
$\Vert (x_n) \Vert_{\ell^q} := \left( \sum_{n=0}^{\infty} x_n^q \right)^{1/q} < \infty,
$ where $q \in (1, \infty)$. We adopt the notation $q$ instead of $p$, as used in~\cite{bauerLpFisherRaoMetricAmari2024}, to avoid confusion with elements $(p_n) \in \Delta$. 

We define $\mathcal{U}_q$ as an open subset of the $\ell^q_{\RR}$ unit sphere $\mathbb{S}(\ell^q_{\RR})$, given by
\[
\mathcal{U}_q := \left\{ (x_n) \in \mathbb{S}(\ell^q_{\RR}) :\ x_n > 0 \quad \forall n \in \mathbb{N} \right\}.
\]
We equip $\Delta$ with a differentiable structure such that the homeomorphism, called the \emph{$q$-root transform}, defined by
\begin{equation}\label{e: def q-root transform}
    \varPhi_q: \Delta \longrightarrow \mathcal{U}_q, \quad (p_n) \mapsto \left(p_n^{1/q}\right),
\end{equation}
becomes a diffeomorphism. For $q \neq 2$, we denote by $\Delta^q$ the set $\Delta$ equipped with this differentiable structure. The tangent space at a point $(p_n) \in \Delta^q$ is given by
\begin{equation*}\label{e: q diff structure on delta}
	T_{(p_n)}\Delta^q := \left\{ (v_n) \in \ell^1_{\mathbb{R}} : \sum_{n=0}^{\infty} v_n = 0 \quad \text{and} \quad \left(v_n p_n^{(1/q - 1)}\right) \in \ell^q_{\mathbb{R}} \right\}.
\end{equation*}
We introduce the $\ell^q$-Fisher–Rao metric, which is a Finsler metric for any $q \in (1, \infty)$, by
\begin{equation}\label{e: def q-Fisher-Rao metric}
    \mathcal{F}^q_{(p_n)}\left( (v_n) \right) := \left( \sum_{n=0}^{\infty} \left\vert \frac{v_n}{p_n} \right\vert^q \cdot p_n \right)^{\frac{1}{q}},
    \quad \forall\ ((p_n), (v_n)) \in T\Delta^q.
\end{equation}
Following a computation along the lines of~\cite[Thm.~3.12]{bauerLpFisherRaoMetricAmari2024}, we obtain the following:
\begin{theorem}\label{T: q-root transform as isom}
    The $q$-root transform $\varPhi_q$ is an isometry between $\left( \Delta^q, \mathcal{F}^q \right)$ and $\left( \mathcal{U}_q, \Vert \cdot \Vert_{\ell^q} \right)$.
\end{theorem}

\begin{remark}\label{r: alpha connections existence}
    The construction of Amari–Čencov $\alpha$-connections in~\cite[Lemma~4.1]{bauerLpFisherRaoMetricAmari2024} and the Chern connection in~\cite[Thm.~4.6]{bauerLpFisherRaoMetricAmari2024} can also be adapted to $(\Delta^q, \mathcal{F}^q)$ using \Cref{T: q-root transform as isom}, but this is omitted here for reasons of space.
\end{remark}

\section{$\ell^2$-Information geometry of an optimization problem}\label{s: 3}

The aim of this section is to use the previously developed geometric framework, $\ell^2$-information geometry, to extend the finite-dimensional information geometry approach from \cite{faybusovich_hamiltonian_1991}, originally used to solve finite-dimensional linear programming problems, to the infinite-dimensional setting. 
We consider the following linear programming (LP) problem on the closure $\bar{\Delta}$ of the $\ell^2$-probability simplex:  
\begin{equation}\tag{LP}\label{eq:ell2_LPP}
	\max_{(p_n)\in \bar{\Delta}} \langle (c_n), (p_n) \rangle_{\ell^2_\R}, \quad (c_n)\in \ell_2^\R.
\end{equation}
By defining the smooth function \( F_{(c_n)}(p_n) := \langle (c_n), (p_n) \rangle_{\ell^2} \) on \( (\Delta, \mathcal{G}) \), the problem \eqref{eq:ell2_LPP} can be understood as finding the maximum of this function on \( \bar{\Delta} \). We now describe the gradient flow lines of \( F_{(c_n)} \) and prove that they converge exponentially fast to solutions of \eqref{eq:ell2_LPP}:

\begin{proposition}
    \label{prop:integral-curves}
Let \( t \mapsto (p_n)(t) \) be a gradient flow line of \( F_{(c_n)} \) with initial value \( (p_n)(0) = (p_n)_0 \in \Delta \). Then, for all \( n \in \mathbb{N} \) and \( t \in [0, \infty) \), the components of \( (p_n)(t) \) are given by
\begin{equation} \label{eq:integral-curves}
	p_n(t) = \frac{p_n(0)\,e^{c_n t}}{\sum_{k=0}^{\infty} p_k(0)\,e^{c_k t}}.
\end{equation}
If in addition the sequence $(c_n)$ in \eqref{eq:ell2_LPP} is strictly monotonic decreasing, then the limit
\[
p_{\max} := \lim_{t \to \infty} (p_n)(t)
\]
exists and is a solution of \eqref{eq:ell2_LPP}.
\end{proposition}
\begin{remark}
  To the authors' best knowledge, an analogue of \Cref{prop:integral-curves} is not known for the space of probability densities on closed manifolds equipped with the Fisher–Rao metric. In an extended version of the present article, this analysis will be carried out in detail, and the case of the space of probability densities on non-compact manifolds without boundary will also be included.
\end{remark}
\begin{proof}
The proof closely follows the argument in \cite[p.~2]{faybusovich_hamiltonian_1991}, with a minor but easily verifiable adaptation. Specifically, the gradient of \( F_{(c_n)}(p_n) := \langle (c_n), (p_n) \rangle_{\ell^2} \) on \( (\Delta, \mathcal{G}) \) satisfies a global Lipschitz condition and is invariant under reparametrizations of the form \( (c_n) \mapsto (c_n - \| (c_n) \|) \). For the sake of completeness, the full details are provided in \Cref{Appendix proof Prop}.
\end{proof}
We close this section with a note of independent interest. In finite-dimensional information geometry, the analogues of the gradient flow lines in \Cref{prop:integral-curves} are known as $e$-geodesics—geodesics with respect to a certain affine connection.

\textbf{Conjecture.} — In analogy with finite-dimensional information geometry \cite{MR1800071,InfGeo_book}, we conjecture that the gradient flow lines in \Cref{prop:integral-curves} can be interpreted as $e$-geodesics with respect to a suitable affine connection on \( \Delta \).

\section{An Infinite dimensional integrable Hamiltonian system}\label{S: 4}
The aim of this section is to explore the Hamiltonian nature of the gradient flows in \Cref{s: 3}  to study the underlying symmetries.  

For this, we denote the space of complex-valued sequences $(z_n) \subset \mathbb{C}$ such that $\| (z_n)\|_{\ell^2}:=\sum_{n=0}^{\infty} |z_n|^2 < \infty$  
by $\ell_{\mathbb{C}}^2$, equipped with the standard Hermitian $\ell^2$-inner product $\langle \cdot, \cdot \rangle_{\ell^2}$,  
which defines an infinite-dimensional Kähler manifold, as  
$\Re \langle \mathrm{i} \cdot, \cdot \rangle_{\ell^2} = \Im \langle \cdot, \cdot \rangle_{\ell^2}$,  
where $\mathrm{i} = \sqrt{-1}$.
The action $\mathbb{S}^1 \curvearrowright \ell_{\mathbb{C}}^2$, where $\S^1=\{e^{\i t}: t\in \RR\}$,  given by $e^{it} \cdot (z_n) := (e^{it}z_n)$ acts by Kähler morphisms and is Hamiltonian, with momentum map 
\begin{equation} \label{eq:S1_momentmap_cpinf}
	\mu_{\mathbb{S}^1} : \ell^2_{\mathbb{C}} \longrightarrow  \mathrm{i} \mathbb{R}, \quad (z_n) \mapsto \mathrm{i} \langle (z_n), (z_n) \rangle_{\ell^2}.
\end{equation}  
Since $\mathrm{i}$ is a regular value of $\mu_{\mathbb{S}^1}$, we obtain by Kähler reduction that the quotient  
$\mu_{\mathbb{S}^1}^{-1}(\mathrm{i}) / \mathbb{S}^1$ is a Kähler manifold. As $\mu_{\mathbb{S}^1}^{-1}(\mathrm{i})$ is precisely the unit sphere $\mathbb{S}(\ell_{\mathbb{C}}^2)$ in $\ell_{\mathbb{C}}^2$, we obtain that  
$\mu_{\mathbb{S}^1}^{-1}(\mathrm{i}) / \mathbb{S}^1 = \mathbb{CP}^{\infty}$, where the Kähler structure we obtained is precisely the Fubini–Study metric $\mathcal{G}^{\mathrm{FS}}$ and the Fubini–Study form $\Omega^{\mathrm{FS}}$. We refer, for example, to \cite[§2]{khesin_geometry_2019} for the explicit form of this Kähler structure. 
We identify the infinite-dimensional torus $\mathbb{T}^{\infty} := \prod_{n=1}^{\infty} \mathbb{S}^1$, where each element in $\mathbb{T}^{\infty}$ is of the form $(e^{\mathrm{i} t_n})$, which inhereits it's topology naturally  as a subgroup of the diagonal unitary operators of $\ell^2_{\CC}$ by identifying each element $(e^{\mathrm{i} t_n})\simeq \operatorname{Diag}((e^{\mathrm{i} t_n}))$. 
We obtain a Hamiltonian action of $\mathbb{T}^{\infty}$ on $\mathbb{CP}^\infty$ through Kähler morphisms given by $\left(e^{\mathrm{i} t_n}\right) \cdot [(z_n)] :=  [((e^{\mathrm{i} t_n})z_n)]$.
The corresponding moment map is  
\begin{equation}\label{eq:Tinf_momentmap_cpinf}
	\mu_{\mathbb{T}^{\infty}}: \mathbb{CP}^\infty \to (\operatorname{Lie}(\mathbb{T}^{\infty}))^*, \quad [(z_n)] \mapsto \frac{\mathrm{i}}{2} (|z_n|^2).
\end{equation}  
Using the identification $(\|z_n\|^2) = (\bar{z}_n)^T \operatorname{Diag}(1) (z_n)$ and noting that $\operatorname{Diag}(1)$ defines a bounded operator, the moment map $\mu_{\mathbb{T}^\infty}$ in~\eqref{eq:Tinf_momentmap_cpinf} takes values in $(\mathrm{i} \ell^2_{\RR})^*$. 
By identifying $\ell^2_{\RR}$ with its dual, we obtain the following commutative diagram: 
\begin{equation}
	\begin{tikzcd}
		\mathbb{S}\left(\ell^2_{\mathbb{C}}\right) \arrow[rr, "\Psi"] \arrow[rd, "\pi"] & & \mathrm{i} \ell^2_{\mathbb{R}} \\
		& \mathbb{T}^\infty \curvearrowright \mathbb{CP}^\infty \arrow[ru, "\mu_{\mathbb{T}^\infty}"] &
	\end{tikzcd}
\end{equation}  
where $\Psi(z) := \frac{\mathrm{i}}{2} (|z_n|^2)$ and $\pi$ denotes the Hopf fibration. Using $\mathcal{U} \subseteq \mathbb{S}\left(\ell^2_{\mathbb{C}}\right)$ and the explicit form of the inverse of the square root map $\varPhi$ in \Cref{t: square root map is isometry},  
we see that the restriction of $\Psi$ to $\mathcal{U}$ is equal to $\frac{2}{\mathrm{i}} \varPhi^{-1}$. Thus, we obtain:
\begin{lemma}\label{l: momentmap image}
	The restriction of $\Psi$ to $\mathcal{U}$ is a diffeomorphism onto $\frac{\mathrm{i}}{2} \Delta$, and the image of the momentum map is given by  
	\[
	\mu_{\mathbb{T}^\infty}(\mathbb{CP}^\infty) = \frac{\mathrm{i}}{2} \Delta.
	\]
\end{lemma}
\begin{remark}
The map \( \mu_{\mathbb{T}^{\infty}} \) can be interpreted as an \(\ell^2\)-version of the inverse of the Madelung transform in the density component, as described in \cite[Prop.~4.3]{khesin_geometry_2019}.
\end{remark}

For a choice of \( (c_n) \in \ell^2_{\RR} \), we define the Hamiltonian
\begin{equation}\label{eq defi Hamiltonian on CP}
    H_{(c_n)} : \mathbb{C}P^\infty \longrightarrow \mathbb{R}, \quad [z_n] \mapsto \left\langle (c_n), 2\bar{\i} \cdot \mu_{\mathbb{T}^\infty}([z_n]) \right\rangle_{\ell^2} = \left\langle (c_n), (|z_n|^2) \right\rangle_{\ell^2}.
\end{equation}
Recall that since \( \Omega^{\mathrm{FS}} \) is a strong symplectic form, the Hamiltonian vector field \( X_{H_{(c_n)}} \) is uniquely determined by the identity $\Omega^{\mathrm{FS}}(X_{H_{(c_n)}}, \cdot) = \mathrm{d}H_{(c_n)}.$
Moreover, since \( (\mathbb{C}P^\infty, \mathcal{G}^{\mathrm{FS}}, \Omega^{\mathrm{FS}}) \) is a Kähler manifold, the gradient and Hamiltonian vector fields of \( H_{(c_n)} \) are related by multiplication with \( \mathrm{i} \).
The flow \( \varphi_{H_{(c_n)}} \) of the vector field \( X_{H_{(c_n)}} \) is called the \emph{Hamiltonian flow} of \( H_{(c_n)} \). This flow preserves the level sets \( H_{(c_n)}^{-1}(\kappa) \) for all levels of the energy \( \kappa \in \mathbb{R} \).  
This Hamiltonian system has two remarkable and surprising properties:
\begin{theorem}\label{t: integrability of Hamiltonian system}
The Hamiltonian system \( (\mathbb{CP}^\infty, \Omega^{\mathrm{FS}}, H_{(c_n)}) \) is completely integrable; that is, it possesses infinitely many Poisson-commuting conserved quantities. \\
Moreover let \( (x_n)(t) \) be a gradient flow line of the restriction of the gradient flow of \( H_{(c_n)} \) on \( (\mathbb{CP}^\infty, \mathcal{G}^{\mathrm{FS}}) \) to \( \mathcal{U} \). Then the limit $p_{\max} := \lim_{t \to \infty} \varPhi\left((x_n)(t)\right)$ exists and solves \eqref{eq:ell2_LPP}.
\end{theorem}
\begin{proof}
 For each \( n \in \mathbb{N} \), we define the sequence $(b_m)$ by 
 \[
b_m:=\begin{cases}
    c_n\, &, \quad\text{for} \quad m=n\\
    0\, &, \quad \text{else}
\end{cases}
 \]
and by this the Hamiltonian
\begin{equation}\label{eq: Hamiltonians on CP inf for integrability}
	H_n : \mathbb{C}P^\infty \longrightarrow \mathbb{R}, \quad [z_m] \mapsto  \sum_{n=0}^{\infty}b_m |z_m|^2= c_n|z_n|^2\, .
\end{equation}
From which we observe that the Hamiltonian flow \( \varphi_{H_n} \) fixes all coordinates except the \( n \)-th, where it acts as a phase rotation.  
Moreover, it is a straightforward exercise in symplectic geometry to verify that all the Hamiltonians \( H_n \) Poisson commute—both with each other and with \( H_{(c_n)} \). The precise details of this construction are carried out in \Cref{Appendix proof of integrability of Hamiltonian system}.\\
By identifying \( \mathcal{U} \) with its image in \( \mathbb{CP}^\infty \), and using \Cref{t: square root map is isometry}, along with the identity $H_{(c_n)} \circ \varPhi = F_{(c_n)}$ on $\Delta$,
where \( (F_{(c_n)}) \) is defined as in the line following \eqref{eq:ell2_LPP}, the proof is complete.
\end{proof}

\section{Further directions}
In upcoming work, the author will generalize the results from \Cref{s: 2} to define the Fisher–Rao metric and the $L^p$ Fisher–Rao metric on the space of probability densities $\mathrm{Dens}(M)$ over non-compact, unbounded manifolds. This generalization will make the $p$-root transform an isometry and will develop the associated geometry in detail. These developments lie beyond the scope of the present paper, as they require working for example with tame Fréchet manifolds.

Independently, it would also be interesting to address more complex optimization problems, as in \Cref{s: 3}, within the framework of $\ell^q$-information geometry developed in \Cref{s: 2}.

We conclude this paper with a speculative remark: since finite-dimensional information geometry has proven to be a powerful tool in reinforcement learning with finitely many actions, it is natural to ask whether $\ell^2$-information geometry or $\ell^p$-information geometry could find similar applications in reinforcement learning with infinitely many actions.
\ \\
\noindent
\textbf{Acknowledgments.}\\
L.M. thanks P. Albers, M. Bleher, J. Cassel, Y. Elshiaty, F. Schlindwein and C. Schnörr for valuable discussions.\\
L.M. acknowledges funding by the Deutsche Forschungsgemeinschaft (DFG, German Research Foundation) – 281869850 (RTG 2229), 390900948 (EXC-2181/1) and 281071066 (TRR 191).\\
L.M. would like to acknowledge the excellent working
conditions and interactions at Erwin Schrödinger International
Institute for Mathematics and Physics, Vienna, during the thematic
programme \emph{``Infinite-dimensional Geometry: Theory and Applications"}
where part of this work was completed.\\
The author gratefully acknowledges the anonymous referees for their insightful comments and suggestions, which have improved the quality of the paper.
\noindent \\
\textbf{Disclosure of Interests.}
The author declare that they have no competing interests.
\appendix 
\section{Proof of \Cref{prop:integral-curves}}\label{Appendix proof Prop}
By applying \Cref{t: square root map is isometry} and noting that the square root map extends to a homeomorphism from the closure $\bar{\Delta}$ onto $\bar{\U}$, we find that \eqref{eq:ell2_LPP} is equivalent to the following optimization problem over the closure $\bar{\U}$ of $\U$ in $\Slr$:
\begin{equation}\tag{NLP}\label{eq: optimization problem sphere}
	\max_{(x_n) \in \bar{\U}} \langle (c_n), (x_n^2) \rangle_{\ell^2_\R}, \quad (c_n )\in \ell_2^\R.
\end{equation}
To analyze this, we define the smooth function \( H_{(c_n)}((x_n)) := \langle (c_n), (x_n)^2 \rangle_{\ell^2} \), where \( (c_n) \in \ell^2_{\mathbb{R}} \) is fixed. The problem \eqref{eq: optimization problem sphere} then reduces to maximizing \( H_{(c_n)} \) over \( \bar{\U} \). The gradient of \( H_{(c_n)} \) in the ambient Hilbert space is given by
\[
\nabla H_{(c_n)} = 2\,\mathrm{Diag}(x_n)(c_n) := 2(x_n c_n).
\]
Projecting this vector onto the tangent space \( T_{(x_n)}\U \) using the orthogonal projection
\[
P_{(x_n)}(v_n) = (v_n) - \langle (v_n), (x_n) \rangle_{\ell^2} (x_n),
\]
yields the Riemannian gradient of \( H_{(c_n)} \) on the sphere \( (\U, \Glt) \):
\begin{equation}\label{eq: gradient on sphere}
	\mathrm{grad}^{\Glt}(H_{(c_n)})(x_n) = P_{(x_n)}\left( 2\,\mathrm{Diag}(x_n)(c_n) \right), \quad \forall (x_n)\in\U.
\end{equation}
By invoking \Cref{t: square root map is isometry} once more, we obtain from \eqref{eq: gradient on sphere} the Riemannian gradient of the function \( F_{(c_n)}(p_n) := \langle (c_n), (p_n) \rangle_{\ell^2} \) on the probability simplex \( (\Delta, \Gfr) \):
\begin{equation} \label{def:W_c}
\mathrm{grad}^{\Gfr}(F_{(c_n)})(p_n) = \operatorname{Diag}(p_n)(c_n) - \langle (p_n), (c_n) \rangle_{\ell^2_\R} (p_n), \quad \forall (p_n) \in \Delta,
\end{equation}
which we denote by \( W_{(c_n)} \) for convenience. It is a standard argument in analysis to verify that \( W_{(c_n)} \) satisfies a global Lipschitz condition. Hence, for any \( (c_n) \in \ell^2_\R \) and initial condition \( (p_n)_0 \in \Delta \), there exists a unique solution to the initial value problem
\begin{equation}\label{eq:IVP}
	\frac{\d}{\d t}(p_n)(t) = W_{(c_n)}(p_n), \quad (p_n)(0) = (p_n)_0 \in \Delta,
\end{equation}
i.e., the gradient flow of \( F_{(c_n)} \) is well-defined on \( \Delta \).\\
Before analyzing the flow lines explicitly, we observe that the vector field \( W_{(c_n)} \) is invariant under reparametrizations of the form \( (c_n) \mapsto (c_n - \| (c_n) \|) \). Another straightforward computation confirms that the curves \( p_n(t) \) given in \eqref{eq:integral-curves} are indeed solutions to \eqref{eq:IVP}.

Furthermore, if the sequence \( (c_n) \) in \eqref{eq:ell2_LPP} is strictly monotonically decreasing, then the solution converges to the corner point
\[
p_{\max} := \lim_{t \to \infty} (p_n)(t) = (1, 0, \dots) \in \bar{\Delta},
\]
which is clearly a solution to \eqref{eq:ell2_LPP}.
\section{Proof of \Cref{t: integrability of Hamiltonian system}}\label{Appendix proof of integrability of Hamiltonian system}
This section provides the details for the Poisson-commuting Hamiltonians stated in \Cref{t: integrability of Hamiltonian system}. We begin by verifying that the Hamiltonians \( H_k \) and \( H_n \), defined in \eqref{eq: Hamiltonians on CP inf for integrability}, Poisson commute for all \( k \neq n \in \mathbb{N} \). It suffices to prove that their lifts \( \hat{H}_k, \hat{H}_n \) to \( \mathbb{S}(\ell^2_{\mathbb{C}}) \), which naturally extend to \( \ell^2_{\mathbb{C}} \), commute and are \( \mathbb{S}^1 \)-invariant.
Recall that the canonical symplectic form on \( \ell^2_{\mathbb{C}} \) is
\[
\omega_{\mathrm{can}} = \frac{\i}{2} \sum_{j=0}^{\infty} \d z_j \wedge \d \bar{z}_j,
\]
inducing the Poisson bracket
\[
\{f, g\} = 2\,\i \sum_{j=0}^{\infty} \left( \frac{\partial f}{\partial \bar{z}_j} \frac{\partial g}{\partial z_j} - \frac{\partial f}{\partial z_j} \frac{\partial g}{\partial \bar{z}_j} \right).
\]
By definition,
\[
\hat{H}_k(z) = c_k |z_k|^2 = c_k z_k \bar{z}_k,
\]
with partial derivatives
\begin{equation*}\label{eq: appendix partial z of H}
    \frac{\partial \hat{H}_k}{\partial z_j} = c_k \bar{z}_k \delta_{jk}, \qquad
    \frac{\partial \hat{H}_k}{\partial \bar{z}_j} = c_k z_k \delta_{jk}, \quad \forall j \in \mathbb{N},
\end{equation*}
and analogously for \( \hat{H}_m \). Since \( k \neq m \), their derivatives have disjoint support, implying
\[
\{ \hat{H}_k, \hat{H}_m \} = 0.
\]
Using \( \mathbb{S}^1 \)-invariance and the moment map description~\eqref{eq:S1_momentmap_cpinf}, it follows that
\[
\{ H_k, H_m \}_{\mathbb{C}P^{\infty}} = 0.
\]
Repeating this argument, we also find that \( H_{(c_n)} \) Poisson-commutes with all \( H_n \), completing the proof.

\bibliographystyle{splncs04}
\bibliography{literature}


\end{document}